\documentclass[11pt]{amsart}
\usepackage{amssymb}
\usepackage{verbatim}
\usepackage{hyperref}
\usepackage{color}

\begin{document}

\newtheorem{theorem}{Theorem}    
\newtheorem{proposition}[theorem]{Proposition}
\newtheorem{conjecture}[theorem]{Conjecture}
\def\theconjecture{\unskip}
\newtheorem{corollary}[theorem]{Corollary}
\newtheorem{lemma}[theorem]{Lemma}
\newtheorem{sublemma}[theorem]{Sublemma}
\newtheorem{observation}[theorem]{Observation}
\theoremstyle{definition}
\newtheorem{definition}{Definition}
\newtheorem{notation}[definition]{Notation}
\newtheorem{remark}[definition]{Remark}
\newtheorem{question}[definition]{Question}
\newtheorem{questions}[definition]{Questions}
\newtheorem{example}[definition]{Example}
\newtheorem{problem}[definition]{Problem}
\newtheorem{exercise}[definition]{Exercise}

\numberwithin{theorem}{section}
\numberwithin{definition}{section}
\numberwithin{equation}{section}

\def\earrow{{\mathbf e}}
\def\rarrow{{\mathbf r}}
\def\uarrow{{\mathbf u}}
\def\varrow{{\mathbf V}}
\def\tpar{T_{\rm par}}
\def\apar{A_{\rm par}}

\def\reals{{\mathbb R}}
\def\torus{{\mathbb T}}
\def\heis{{\mathbb H}}
\def\integers{{\mathbb Z}}
\def\naturals{{\mathbb N}}
\def\complex{{\mathbb C}\/}
\def\distance{\operatorname{distance}\,}
\def\support{\operatorname{support}\,}
\def\dist{\operatorname{dist}\,}
\def\Span{\operatorname{span}\,}
\def\degree{\operatorname{degree}\,}
\def\kernel{\operatorname{nullspace}\,}
\def\dim{\operatorname{dim}\,}
\def\codim{\operatorname{codim}}
\def\trace{\operatorname{trace\,}}
\def\Span{\operatorname{span}\,}
\def\dimension{\operatorname{dim}\,}
\def\codimension{\operatorname{codimension}\,}
\def\nullspace{\scriptk}
\def\ZZ{ {\mathbb Z} }
\def\p{\partial}
\def\rp{{ ^{-1} }}
\def\Re{\operatorname{Re\,} }
\def\Im{\operatorname{Im\,} }
\def\ov{\overline}
\def\eps{\varepsilon}
\def\lt{L^2}
\def\diver{\operatorname{div}}
\def\curl{\operatorname{curl}}
\def\etta{\eta}
\newcommand{\norm}[1]{ \|  #1 \|}
\def\expect{\mathbb E}

\newcommand{\Norm}[1]{ \left\|  #1 \right\| }
\newcommand{\set}[1]{ \left\{ #1 \right\} }
\def\one{\mathbf 1}
\def\whole{\mathbf V}
\def\snarl{\mathfrak S}
\newcommand{\modulo}[2]{[#1]_{#2}}

\def\scriptf{{\mathcal F}}
\def\scriptg{{\mathcal G}}
\def\scriptm{{\mathcal M}}
\def\scriptb{{\mathcal B}}
\def\scriptc{{\mathcal C}}
\def\scriptt{{\mathcal T}}
\def\scripti{{\mathcal I}}
\def\scripte{{\mathcal E}}
\def\scriptv{{\mathcal V}}
\def\scriptw{{\mathcal W}}
\def\scriptu{{\mathcal U}}
\def\scriptS{{\mathcal S}}
\def\scripta{{\mathcal A}}
\def\scriptr{{\mathcal R}}
\def\scripto{{\mathcal O}}
\def\scripth{{\mathcal H}}
\def\scriptd{{\mathcal D}}
\def\scriptl{{\mathcal L}}
\def\scriptn{{\mathcal N}}
\def\scriptp{{\mathcal P}}
\def\scriptk{{\mathcal K}}
\def\frakv{{\mathfrak V}}

\author{Michael Christ}
\address{
        Michael Christ\\
        Department of Mathematics\\
        University of California \\
        Berkeley, CA 94720-3840, USA}
\email{mchrist@math.berkeley.edu}
\thanks{The author was supported in part by NSF grant
DMS-0901569.} 

\date{May 27, 2010.}

\title[ Multilinear Oscillatory Integrals]
{Multilinear Oscillatory Integrals 
\\ Via Reduction of Dimension}

\begin{abstract}
Dimensional restrictions in a theorem of Christ, Li, Tao, and Thiele
on multilinear oscillatory integral forms can be relaxed.
\end{abstract}

\maketitle

\section{Introduction}


By a multilinear oscillatory integral we mean a complex scalar-valued multilinear 
form $(f_1,\cdots,f_{n})\mapsto I(P;f_1,\cdots,f_{n})$ 
defined by an integral expression
\begin{equation}
I(P;f_1,\cdots,f_{n}) = \int_{\reals^m} e^{iP(x)}\prod_{j=1}^{n} f_j(\pi_j(x))\,dx.
\end{equation}
This expression involves parameters
$m,n,(\kappa_1,\cdots,\kappa_{n}),(\pi_1,\cdots,\pi_{n})$
where $m$ is the ambient dimension,
$\pi_j:\reals^m\to\reals^{\kappa_j}$ are surjective linear transformations,
and $1\le \kappa_j\le m-1$.
Each function $f_j$ is assumed to belong to $L^\infty(\reals^{\kappa_j})$,
and to have support in a specified compact set $B_j\subset\reals^{\kappa_j}$.
Here $n\ge 2$, $m\ge 2$. 
The phase function $P$ will always be assumed to be a real-valued polynomial.

In this note we continue the study, initiated in \cite{cltt}, of inequalities
of the form
\begin{equation} \label{lambdadecay}
|I(\lambda P;f_1,\cdots,f_{n})|\le C(1+|\lambda|)^{-\rho}\prod_j\norm{f_j}_{L^\infty},
\end{equation}
where $\lambda\in\reals$ is arbitrary, while $C,\rho\in\reals^+$
are constants which are permitted to depend on $P$ and on the supports of $\{f_j\}$.
In the ``linear'' case $n=2$, there is an extensive literature concerning
such inequalities, typically phrased in terms
of $\prod_j\norm{f_j}_{L^{p_j}}$ for more general exponents $p_j$. See for instance
\cite{stein} for an introduction.
Much less is known concerning the multilinear case $n\ge 3$.

A central notion, investigated in \cite{cltt} and \cite{sublevel}, is
that of nondegeneracy of the phase.
A polynomial $P$ is said to be degenerate relative to $\{\pi_j\}$
if $P$ can be decomposed as $\sum_j Q_j\circ\pi_j$,
for some polynomials $Q_j$. Various forms of this condition
are equivalent; in particular, if $P$ has degree $D$,
then $P$ admits a decomposition $P=\sum_j \pi_j^*(h_j)$
where $h_j$ are distributions on $\reals^{\kappa_j}$ and $\pi_j^*$
is the natural pull back operation,
if and only if $P$ admits a decomposition $P=\sum_j Q_j\circ\pi_j$
where each $Q_j$ is a polynomial of degree $\le D$.
$P$ is said to be nondegenerate, relative to $\{\pi_j\}$,
if it is not degenerate.

Whenever $\pi_j,\tilde\pi_j$ are surjective mappings
with identical nullspaces and with ranges of equal dimensions,
$\tilde\pi_j=L\circ\pi_j$ for some linear transformation $L$.
Therefore nondegeneracy is a property only
of the collection of subspaces $\scriptv_j=\kernel(\pi_j)$,
rather than of the mappings $\pi_j$, so we may equivalently speak
of nondegeneracy relative to a collection of subspaces $\{\scriptv_j\}$.

Let $D\ge 1$ be a positive integer,
and fix $\{\scriptv_j=\kernel(\pi_j)\}$.
The vector space of all degenerate polynomials $P:\reals^m\to\reals$
of degree $\le D$
is a subspace $\scriptp_{\text{degen}}$
of the vector space $\scriptp(D)$ of all polynomials $P:\reals^m\to\reals$
of degree $\le D$.
Denote the quotient space by 
$\scriptp(D)/\scriptp_{\text{degen}}$,
by $[P]$ the equivalence class of $P$ in 
$\scriptp(D)/\scriptp_{\text{degen}}$,
and by
$\norm{\cdot}_{\text{ND}}$
some fixed choice of norm for the quotient space.

A family of subspaces $\scriptv_j\subset\reals^m$
of codimensions $\kappa_j$
is said to have the {\em uniform power decay} property
if for each degree $D$
there exists an exponent $\gamma>0$
such that for any linear mappings $\pi_j$
with nullspaces equal to $\scriptv_j$,
and for any collection of bounded subsets $B_j\subset\reals^{\kappa_j}$,
there exists $C<\infty$ such that
whenever each $f_j$ is supported in $B_j$,
\begin{equation} \label{decaydef1}
|I(P;f_1,\cdots,f_{n})|\le C\norm{P}_{\text{ND}}^{-\gamma} 
\prod_{j=1}^{n}\norm{f_j}_{L^\infty}.
\end{equation}

Certain variations on this definition are also natural. One can
consider only one-parameter families of polynomials
$\{\lambda P_0: \lambda\in\reals\}$, where $P_0$ remains fixed.
One might allow the exponent $\gamma$ to depend on the supports $B_j$;
this would be a more natural hypothesis in an extension to nonpolynomial $C^\infty$
phases $P$. The case of polynomial phases $P$,
with bounds which depend only on $\norm{P}_{\text{ND}}$, is 
fundamental, so we restrict to this case in this paper. For polynomial phases, the methods 
of \cite{cltt} and of this paper show that $\gamma$ can be taken to be independent of $\{B_j\}$.

The uniform decay property is defined in the same way, with
$\norm{P}_{\text{ND}}^{-\gamma}$
replaced by
$\Theta(\norm{P}_{\text{ND}})$
for some function satisfying $\Theta(R)\to 0$ as $R\to\infty$.
Nondegeneracy is a necessary condition even for 
a yet weaker form of the decay property \cite{cltt}.
No other necessary conditions are known to this author.

In the nonsingular case in which the mapping $\reals^m\owns x\mapsto(\pi_j(x))_{j=1}^{n}
\in\times_{j=1}^{n}\reals^{\kappa_j}$ is bijective,
it has been shown by Phong and Stein that
$P$ is nondegenerate relative to $\{\kernel(\pi_j)\}$
if and only if \eqref{lambdadecay} holds;
in that case, nondegeneracy admits a simple characterization in terms
of nonvanishing of some mixed partial derivative of $P$.
The singular case, where this embedding is not bijective, 
is the object of our investigation.
As is explained in \cite{sublevel}, the singular situation only genuinely
arises for $n\ge 3$.

It was shown in \cite{cltt} that the uniform power decay property holds in two primary cases:
firstly, when $\kappa_j=m-1$ for all $j$, and secondly,
when $\kappa_j=1$ for all $j$ and $n<2m$, provided in this second case
that $\{\kernel(\pi_j)\}$ is in general position. 
It was subsequently proved in \cite{sublevel}
that certain uniform upper bounds for measures of sublevel sets,
bounds which would be implied by the uniform decay property,
are valid for all $\{\pi_j\}$, subject only to the hypothesis that it is possible
to choose coordinates in $\reals^m$ and in all $\reals^{\kappa_j}$
in which all $\pi_j$ are represented by matrices with rational entries.
In that result the rate of decay proved to hold was not
of the form of a negative power of $\norm{P}_{\text{ND}}$, but merely some slowly
decaying function; the proof relied on a strong form of Szemer\'edi's theorem.

This note extends the second result of \cite{cltt}
to more general codimensions. 
\begin{theorem} \label{thm:reduction}
If a finite family of subspaces
$\{\scriptv_\alpha\}$ of $\reals^m$ of codimensions $\kappa_\alpha\in [1,m-1]$
is in general position and  satisfies
\begin{equation} \label{newhypothesis}
2\max_\beta\kappa_\beta + \sum_{\alpha}\kappa_\alpha\le 2m,
\end{equation}
then $\{\scriptv_\alpha\}$ has the uniform power decay property. 
\end{theorem}

The coefficient of $2$ in \eqref{newhypothesis} is unnatural,
and the proof still applies in many cases with $2\max_\beta\kappa_\beta$
replaced by $\max_\beta\kappa_\beta$, or even a smaller quantity, but
it seems difficult to formulate a simple general result.
When all $\kappa_j=1$, the hypothesis \eqref{newhypothesis} reduces
to $n\le 2m-2$, whereas the hypothesis $n\le 2m-1$ actually suffices by \cite{cltt}.

It remains to define the notion of general position in this theorem. 
The following notation will be useful in that regard.
\begin{definition}
Let $\whole$ be a real vector space of some dimension $m\ge 2$.
For any index set $A$ and any $A$-tuple $(\kappa_\alpha: \alpha\in A)\in [1,m-1]^A$,
$G(\whole,A,(\kappa_\alpha: \alpha\in A))$ denotes
the manifold consisting of all $|A|$-tuples of linear subspaces
of $\whole$ of codimensions $\kappa_\alpha$.
An element of $G(\whole,A,(\kappa_\alpha: \alpha\in A))$ will be called a snarl.
\end{definition}

We will sometimes set $A=\{1,2,\cdots,n\}$ and identify $\whole$ with $\reals^m$,
and write 
$G(m,A,(\kappa_\alpha: \alpha\in A))$,
or
instead $G(m,\kappa_j: 1\le j\le n)$,
to simplify notation.
$G(m,A,(\kappa_\alpha: \alpha\in A))$
is a product of standard Grassmann manifolds $G(m,\kappa_\alpha)$,
and  thus carries a natural real analytic structure.

A precise statement of 
Theorem~\ref{thm:reduction} is that whenever $m,(\kappa_\alpha)$
satisfy \eqref{newhypothesis}, there exists an analytic 
subvariety $X\subset 
G(m,A,(\kappa_\alpha: \alpha\in A))$
of positive codimension,
such that 
every snarl in the complement of $X$
has the uniform decay property.
We will not describe $X$ explicitly, for to do so would be prohibitively complicated,
but it is constructed in principle through a recursive procedure defined in the proof of the theorem.
However, in the special case where every subspace $\scriptv_j$ has codimension one,
an explicit definition of general position is given in Definition~\ref{defn:specialgeneral}.
Whenever we speak of general position with all $\kappa_j=1$,
it is understood that we refer to that explicit definition.


The proof of Theorem~\ref{thm:reduction} proceeds by induction on the
codimensions $\kappa_j$, which reduces the general case to the case
where all codimensions equal one, already treated in \cite{cltt}. 
In a companion paper \cite{christdosilva}, a limited class of
special cases of Theorem~\ref{thm:reduction} is treated by a rather different method,
which we believe to be of interest despite its currently more restricted scope.

The symbols $C,c$ will denote constants in $(0,\infty)$, whose values are permitted
to change from one occurrence to the next. They typically depend only on $m,n$,
$\{\pi_j\}$, an upper bound for the degree of the polynomial phase $P$,
and the supports $B_j$ of $f_j$.
$\langle x\rangle$ is shorthand for $(1+|x|^2)^{1/2}$.

The author thanks Diogo Oliveira e Silva for useful corrections and comments on the exposition.

\section{An Example}
Heavy notation in the general discussion below obscures a straightforward idea,
so we discuss here a simple example, in the hope of illuminating the proof.
Consider
\[
\iint_{\reals^4} e^{iP(x_1,x_2,y_1,y_2)}f_0(x_1,y_1)f_1(x_2,y_2)f_2(x_1+x_2,y_1+y_2)
\,dx_1\,dx_2\,dy_1\,dy_2.
\]
Rewrite this as
\[
\iint\Big(\iint
e^{iP(s,u,t,-t+v)}f_0(s,t)f_1(u,-t+v)f_2(s+u,v)\,ds\,dt\Big)
\,du\,dv.
\]
The inner integral can be rewritten as
\[
\iint f_0(s,t)\cdot e^{iQ_{u,v}(s,t)}F_{1,u,v}(t)F_{2,u,v}(s+t)\,ds\,dt
=
\Big\langle e^{iQ_{u,v}}(F_{1,u,v}\circ L_1)(F_{2,u,v}\circ L_2),\, \overline{f_0} \Big\rangle
\]
where $F_{1,u,v}(t)=f_1(u,-t+v)$, $F_{2,u,v}$ has a similar expression in
terms of $f_2$, $Q_{u,v}$ is a certain polynomial in $(s,t)$,
$L_1(s,t)=t$, and $L_2(s,t)=s+t$.

If the $4$-fold integral is not suitably small, then
there exists $(u_0,v_0)$ for which this inner product is not suitably small.
Therefore $f_0=f_0(s,t)$ has a nonnegligible inner product with a function of
a special form, namely, a product of a function of $L_1(s,t)$, a function of $L_2(s,t)$,
and a polynomial $q(s,t)$ whose degree does not exceed that of $P$.

By an argument used in \cite{cltt} (see the derivation of \eqref{swallowingresult} below),
it suffices to analyze the case where $f_0$ is {\em equal} to such
a product. Substitute this product back into the original integral over $\reals^4$.
Then $P$ is replaced by $\tilde P = P(x_1,x_2,y_1,y_2)+q(x_1,y_1)$. As a function of $(x_1,x_2,y_1,y_2)$,
$q(x_1,y_1)$ is degenerate. Therefore $\tilde P$ belongs to the same equivalence class
as $P$.

The effect is a reduction to the case where $f_0(x_1,y_1)$
is replaced by a product of two factors, each of which depends only on
the image of $(x_1,x_2,y_1,y_2)$ under a mapping $L_j$.
The same reasoning can be applied to similarly reduce $f_1,f_2$.
There results a multilinear form involving $6$ functions  $g_\alpha(L_\alpha(x_1,x_2,y_1,y_2))$, 
where each $L_\alpha$ is a linear mapping from $\reals^4$ to $\reals^1$,
rather than to the original $\reals^2$.
The case of one-dimensional target spaces was treated in \cite{cltt}.

In \S\ref{section:resolution} we will formalize 
the concept of a resolution of a snarl, 
a sequence of moves which, in the example just presented,
transforms the given collection of three subspaces of codimension two
into a collection of six subspaces of codimension 1.
In \S\ref{section:linearalgebra}
we will prove that any snarl in general position
admits a resolution by a sequence of such moves.
Finally, in \S\ref{section:slicing}, we will carry out the analytic argument outlined
in the preceding paragraphs to demonstrate
that each move preserves the uniform power decay property.

\section{Resolution} \label{section:resolution}

\begin{definition}
A splitting of a snarl 
$(\whole,A,\{\scriptv_\alpha: \alpha\in A\})$
is a snarl
$(\whole,B,\{\scriptw_\beta: \beta\in B\})$
with index set $B$ satisfying
$|B|=|A|+1$, $|A\cap B|=|A|-1$,
if $\alpha\in A\cap B$ then $\scriptw_\alpha=\scriptv_\alpha$,
and if indices $\alpha_0,\beta',\beta''$ are specified so that
$B\setminus A = \{\beta',\beta''\}$ and $A\setminus B=\{\alpha_0\}$,
then 
\begin{gather*}
\scriptw_{\beta'}\cap\scriptw_{\beta''}=\scriptv_{\alpha_0}
\\
\codim(\scriptw_{\beta'}) + \codim(\scriptw_{\beta''})
= \codim(\scriptv_{\alpha_0}).
\end{gather*}
\end{definition}

Direct consequences of the definition are
\begin{align}
&\sum_{\alpha\in A}\codim(\scriptv_\alpha)
=
\sum_{\beta\in B}\codim(\scriptw_\beta).
\\
&\max_\alpha\codim(\scriptv_\alpha)\ge
\max_\beta\codim(\scriptw_\beta).
\end{align}
Therefore if a snarl satisfies our main hypothesis \eqref{newhypothesis},
any splitting continues to satisfy that hypothesis.

If 
$(\whole,A,\{\scriptv_\alpha: \alpha\in A\})$
is a snarl with index set $A$,
then for any nonempty subset $A'\subset A$,
$\scriptv_{A'}$ is defined
to be $\cap_{\alpha\in A'} \scriptv_\alpha$.
Let
$(\whole,B,\{\scriptw_\beta: \beta\in B\})$
be a splitting of a snarl
$(\whole,A,\{\scriptv_\alpha: \alpha\in A\})$.
Let $\beta',\beta'',\alpha_0$ be the three distinguished indices
which appear in the preceding definition.
\begin{definition}
A splitting 
$(\whole,B,\{\scriptw_\beta: \beta\in B\})$
of a snarl
$(\whole,A,\{\scriptv_\alpha: \alpha\in A\})$
is transverse
if $A\setminus \{\alpha_0\}$ can be partitioned as the disjoint union of two nonempty
sets $A',A''$
such that 
\begin{gather*}
\dimension(\scriptw_{\beta'}\cap\scriptv_{A'})>0,
\\
\dimension(\scriptw_{\beta''}\cap\scriptv_{A''})>0,
\\
\whole = 
\scriptw_{\beta'} + \scriptw_{\beta''}, 
\\
\scriptw_{\beta'}  + \scriptv_{\alpha_0}
\text{ and } 
\scriptw_{\beta''} + \scriptv_{\alpha_0}
 \text{ are proper subspaces of } \whole.
\end{gather*}
\end{definition}

In \S\ref{section:slicing} we will establish:
\begin{proposition} \label{prop:induction}
Suppose that the snarl
$\snarl^\sharp$
is a transverse splitting of a snarl
$\snarl$.
If 
$\snarl^\sharp$
has the uniform power decay property,
then
so does $\snarl$.
\end{proposition}

\begin{definition}
A chain of transverse splittings of a snarl
$\snarl$
is a finite sequence of snarls
$(\snarl_k)_{k=0}^N$
such that
$\snarl_0=\snarl$,
and
$\snarl_{k+1}$
is a transverse splitting of
$\snarl_k$
for each $k\in\{0,1,2,\cdots,N-1\}$.
\end{definition}

\begin{definition}
A snarl $(\whole,A,\{\scriptv_\alpha: \alpha\in A\})$
one-dimensional
if for every $\alpha\in A$, $\scriptv_\alpha$ has codimension one. 
\end{definition}

\begin{definition}
A resolution 
$(\snarl_k)_{k=0}^N$
of a snarl
$\snarl$
is a chain of transverse splittings of
$\snarl$
such that 
$\snarl_N$ is one-dimensional.
$\snarl_N$
is called the terminal element of this resolution.
\end{definition}

\begin{definition} \label{defn:specialgeneral}
A one-dimensional snarl 
$(\whole,A,\{\scriptv_\alpha: \alpha\in A\})$
is said to be in general position if 
for any index set $A'\subset A$,
$\{\scriptv_\alpha: \alpha\in A'\}$
spans a subspace of dimension $\min(|A'|, m)$.
\end{definition}

It was shown in Theorem~2.1 of \cite{cltt}
that any one-dimensional snarl 
in $\reals^m$ with index set $A$
satisfying $|A|<2m$ has the uniform power decay property,
provided that it is in general position in this sense.
Combining that theorem with Proposition~\ref{prop:induction} gives:
\begin{proposition} \label{prop:resolutionsuffices}
Let 
$\snarl=(\whole,A,\{\scriptv_\alpha: \alpha\in A\})$
be a snarl satisfying
\[
\max_{\alpha\in A}\codim(\scriptv_\alpha)
+ \sum_{\alpha\in A}\codim(\scriptv_\alpha)\le 2\dimension(\whole).
\]
Suppose that
$\snarl$
admits a resolution with terminal element
in general position.
Then 
$\snarl$
has the uniform power decay property.
\end{proposition}

There remains the question of the existence and abundance of snarls admitting
resolutions with the desired properties.
\begin{proposition} \label{prop:abundance}
Fix $m>1$, a finite index set  $A$, and $\{\kappa_\alpha: \alpha\in A\}$ 
satisfying  \eqref{newhypothesis}.
There exists an analytic variety $X$ of positive codimension
in $G(m,(\kappa_\alpha: \alpha\in A))$,
such that any snarl $Q\notin X$
admits a resolution with terminal element in general position.
\end{proposition}

Propositions~\ref{prop:resolutionsuffices} and \ref{prop:abundance}
together establish our main theorem.
By a straightforward induction, Proposition~\ref{prop:abundance}
is a consequence of the following result.
\begin{proposition} \label{prop:atlast}
Let $m,n$ and an index set $A$ of cardinality $n$ be given.
Let $\{\kappa_\alpha: \alpha\in A\}$ satisfy \eqref{newhypothesis}.
There exist an index set $B=(A\setminus\{\alpha_0\})\cup\{\beta',\beta''\}$
of cardinality $n+1$
and parameters $\{\kappa_{\beta'},\kappa_{\beta''}\}$
such that $\{\kappa_\beta: \beta\in B\}$ continues to satisfy
\eqref{newhypothesis}, and
such that for any 
analytic variety $Y\subset G(m,(\kappa_\alpha: \alpha\in B))$ of positive codimension,
there exists an
analytic variety $X\subset G(m,(\kappa_\alpha: \alpha\in A))$ of positive codimension
such that any snarl  in
$G(m,(\kappa_\alpha: \alpha\in A))\setminus X$ 
admits a transverse splitting belonging to
$G(m,(\kappa_\alpha: \alpha\in B))\setminus Y$. 
\end{proposition}

A defect of our theory is that
the variety $X$ in Proposition~\ref{prop:abundance} has been defined not explicitly,
but only by a rather complicated recursive procedure.
However, Proposition~\ref{prop:resolutionsuffices}
can be applied directly to any snarl for which a resolution can be found.

\section{Proof of Proposition~\ref{prop:atlast}} \label{section:linearalgebra}

Identify $A$ with $\{0,1,\cdots,n-1\}$
in such a way that $\kappa_{0}=\max_j\kappa_j$.
Partition the set of indices $\{1,2,\cdots,n-1\}$
into $2$ nonempty disjoint subsets, $S',S''$.
Consider 
\begin{alignat*}{2}
&\scriptv_{S'}=\cap_{j\in S'}\scriptv_j
\qquad\qquad
&&
\scriptv_{S''}=\cap_{j\in S''}\scriptv_j
\\
&\kappa_{S'}=\sum_{j\in S'}\kappa_j
&&
\kappa_{S''}=\sum_{j\in S''}\kappa_j.
\end{alignat*}
Choose this partition so that
$|\kappa_{S'}-\kappa_{S''}|\le\kappa_0$, which is possible because
$\kappa_0\ge\kappa_j$ for all $j$. 

\begin{lemma} \label{lemma:constructW}
Suppose that $\sum_{j=0}^{n-1}\kappa_j<2m$,
and that $\max_j\kappa_j>1$.
There exist integers 
$\kappa',\kappa''\in \{1,2,\cdots,\kappa_0-1\}$,
depending only on $m$ and on $\{\kappa_j: 0\le j<n\}$ and
satisfying $\kappa'+\kappa''=\kappa_0$,
together with an analytic variety $X_0\subset G(m,\kappa_j: 0\le j<n)$ of positive codimension,
such that whenever $(\scriptv_j:0\le j<n)\notin X_0$,
there exist
subspaces $W'\subset\scriptv_{S'}$ and $W''\subset\scriptv_{S''}$
of dimensions $\kappa',\kappa''$ respectively,
which satisfy
\begin{align}
&W'\cap W''=\{0\}
\\
&(W'+W'')\cap\scriptv_0=\{0\}.
\end{align}
\end{lemma}

If $W',W''$ satisfy these conclusions,
define $\scriptv_n=\scriptv_0+W''$ and
$\scriptv_{n+1}=\scriptv_0+W'$.
Then $(\scriptv_i: 1\le i\le n+1)$ is a transverse splitting of
$(\scriptv_j: 0\le j<n)$. 

\begin{proof}
Suppose without loss of generality that $\kappa_{S'}\ge \kappa_{S''}$.
Since $\kappa_{S'}\le \kappa_{S''}+\kappa_0$
and $\kappa_{S'}+\kappa_{S''}<2m-\kappa_0$,
$2\kappa_{S'}\le \kappa_{S'}+(\kappa_{S''}+\kappa_0)<2m$.
Therefore
$\max(\kappa_{S'},\kappa_{S''})<m$.
If $\{\scriptv_j\}$ is in general position,
\begin{align*}
&\dimension(\scriptv_{S'})
=\max(0, m-\kappa_{S'}) =m-\kappa_{S'}\ge 1
\\
&\dimension(\scriptv_{S''})=\max(0,m-\kappa_{S''}) =m-\kappa_{S''}\ge 1
\end{align*}
and since $2m-\kappa_{S'}-\kappa_{S''}-\kappa_0\ge 0$ by \eqref{newhypothesis}, 
if $(\scriptv_j)$ is in general position then
\begin{equation} \label{tritransverse}
\dimension\big(\scriptv_{S'}
+\scriptv_{S''}
+\scriptv_0\big)
=\min\big(m,m-\kappa_{S'}+m-\kappa_{S''}+m-\kappa_0\big)
=m.
\end{equation}

Furthermore,
since $\scriptv_0$ has positive codimension
and $\scriptv_{S'},\scriptv_{S''}$ have positive dimensions,
if $(\scriptv_i: 0\le i<n)$ is in general position, then
neither of $\scriptv_{S'},\scriptv_{S''}$ is contained
in $\scriptv_0$.
Moreover, $\scriptv_0$ has codimension $\kappa_0\ge 2$.
These facts, together with \eqref{tritransverse}, ensure that there exist
$\kappa',\kappa''\in[1,\kappa_0]$ satisfying $\kappa'+\kappa''=\kappa_0$,
and subspaces $W'\subset\scriptv_{S'}$ and $W''\subset\scriptv_{S''}$
of dimensions $\kappa',\kappa''$ respectively,
such that $W'\cap W''=\{0\}$ 
and 
\begin{equation} \label{tridirectsum}
W'+W''+\scriptv_0=\reals^m.
\end{equation}
Since 
\[\dimension(W')+ \dimension(W'')+\dimension(\scriptv_0)
= \kappa'+\kappa'' + (m-\kappa_0)=m,\]
\eqref{tridirectsum} is a direct sum decomposition.

Fix such $\kappa',\kappa''$.
Choose subspaces $U',U''\subset\reals^m$ of codimensions $m-\kappa_{S'}-\kappa'$ and
$m-\kappa_{S''}-\kappa''$ respectively, which are transverse to one another. 
To an arbitrary $\snarl=(\scriptv_j: 0\le j<n)\in G(m,\kappa_j: 0\le j<n)$
associate $W'(\snarl)=U'\cap\scriptv_{S'}$ and $W''(\snarl) = U''\cap\scriptv_{S''}$.
The set of all $\snarl$ for which they fail to do so, is
an analytic variety $X_0$ of positive codimension.
\end{proof}

The hypothesis $\sum_{j=0}^{n-1}\kappa_j<2m$ is not sufficient to ensure that
the splitting $(\scriptv_j: 1\le j\le n+1)$ lies in general position.
Indeed,
the sum of the dimensions of 
$\scriptv_n\cap\scriptv_{S'}$
and
$\scriptv_{n+1}\cap\scriptv_{S''}$ 
is required
by the above construction to be $\ge\kappa_0$.
For $(\scriptv_i: 1\le i\le n+1)$ in general position,
these two intersections will have dimensions equal to
$m-\kappa_{S'}-\kappa_n, m-\kappa_{S''}-\kappa_{n+1} $, respectively.
Thus the construction requires
$2m-\kappa_{S'}-\kappa_{S''}-\kappa_n-\kappa_{n+1}\ge \kappa_0$.
Since $\kappa_n+\kappa_{n+1}=\kappa_0$,
this is equivalent to
$2m-\sum_{j=0}^{n-1}\kappa_j\ge \kappa_0$,
that is, to $\max_i\kappa_i+\sum_j\kappa_j\le 2m$.


\begin{lemma}
Let $(\kappa_j: 0\le j<n)$ satisfy 
$2\max_i\kappa_i+ \sum_{j}\kappa_j\le 2m$
and $\max_i\kappa_i>1$.
There exist $\kappa_n,\kappa_{n+1}\in[1,\kappa_0-1]$
satisfying $\kappa_n+\kappa_{n+1}=\kappa_0$ with the following property.
For any analytic subvariety $Y\subset
G(m,\kappa_i: 1\le i\le n+1)$
of positive codimension, 
there exists an analytic subvariety $X\subset G(m,\kappa_j: 0\le j<n)$
of positive codimension,
such that if $(\scriptv_j: 0\le j<n)\notin X$,
then in the above construction, $W',W''$ can be chosen so that
$(\scriptv_1,\cdots,\scriptv_{n-1},\scriptv_0+W',\scriptv_0+W'')\notin Y$.
\end{lemma}

\begin{proof}
Choose $S',S''$ as above, so that
$|\kappa_{S'}-\kappa_{S''}|\le \kappa_0$.
Then
$(m-\kappa_{S'})+(m-\kappa_{S''})\ge 3\kappa_0$
by \eqref{newhypothesis}  and
the choice $\kappa_0=\max_{j}\kappa_j$.
Therefore $m-\kappa_{S'}$ and $m-\kappa_{S''}$ are both $\ge\kappa_0$;
it is here that the full strength of \eqref{newhypothesis} is used.
Consequently if $\kappa_n,\kappa_{n+1}\in[1,\kappa_0-1]$ 
are chosen to satisfy $\kappa_n+\kappa_{n+1}=\kappa_0$, then
\begin{gather} 
\label{eq:allof2}
m-\kappa_{S'}-\kappa_n\ge\kappa_{n+1}
\\
\label{eq:allof3}
m-\kappa_{S''}-\kappa_{n+1}\ge\kappa_{n}.
\end{gather}


\medskip
Consider any 
$\snarl^\sharp=(\scriptv_j: 1\le j\le n+1)\in G(m,n+1,\kappa_1,\cdots,\kappa_{n+1})$ in general position,
where the precise meaning of general position remains to be specified.
Define $\scriptv_0=\scriptv_n\cap\scriptv_{n+1}$.
Then
$\dimension(\scriptv_n\cap\scriptv_{n+1})=\max(0,m-\kappa_n-\kappa_{n+1})=m-\kappa_0$,
so $\scriptv_0=\scriptv_n\cap\scriptv_{n+1}$ has codimension $\kappa_0$.
Moreover, general position ensures that
\begin{align*}
\dimension(\scriptv_{S'}\cap\scriptv_n)&=m-\kappa_{S'}-\kappa_n
\\
\dimension(\scriptv_{S''}\cap\scriptv_{n+1})&=m-\kappa_{S''}-\kappa_{n+1}.
\end{align*}
Therefore if $\snarl^\sharp$ is in general position, 
\begin{equation*}
\dimension(\scriptv_n\cap\scriptv_{S'})
+\dimension(\scriptv_{n+1}\cap\scriptv_{S''})
+\dimension(\scriptv_n\cap\scriptv_{n+1})
\ge \kappa_n+\kappa_{n+1}+(m-\kappa_0)=m.
\end{equation*}
Since the index sets $S',S'',\{n,n+1\}$ are pairwise disjoint, 
general position then implies that
$ (\scriptv_n\cap\scriptv_{S'}) 
+(\scriptv_{n+1}\cap\scriptv_{S''})
+(\scriptv_n\cap\scriptv_{n+1})=\reals^m$.

The two subspaces
$\scriptv_n\cap\scriptv_{S'}$ 
and $\scriptv_0=\scriptv_n\cap\scriptv_{n+1}$ 
are contained in $\scriptv_n$ and have dimensions
$m-\kappa_{S'}-\kappa_n$ and $m-\kappa_0$, respectively. 
If $\{\scriptv_j: j\in S'\}$ and $\scriptv_{n+1}$
are jointly in general position relative to $\scriptv_n$,
these two subspaces will be transverse;
their sum will have dimension equal to 
\begin{multline*}
\max(m-\kappa_n, m-\kappa_{S'}-\kappa_n + m-\kappa_0)
= 2m-\kappa_{S'}-\kappa_n-\kappa_0
\\
\ge m-\kappa_0+\kappa_{n+1}
=\dimension(\scriptv_0)+\kappa_{n+1},
\end{multline*}
using \eqref{eq:allof2}.
Therefore there exists a subspace $W'\subset\scriptv_n\cap\scriptv_{S'}$
of dimension exactly $\kappa_{n+1}$, satisfying
$\dimension(W'+\scriptv_0) = \dimension(W')+\dimension(\scriptv_0)$.
Since $W',\scriptv_0$ are both contained in $\scriptv_n$ and
the sum of their dimensions equals the dimension $\kappa_{n+1}+m-\kappa_0
=m-\kappa_n$ of $\scriptv_n$, their span equals $\scriptv_n$.
For the same reasons,
there exists a subspace $W''\subset\scriptv_{n+1}\cap\scriptv_{S''}$
of dimension $\kappa_{n}$ 
which is transverse to $\scriptv_0$,
such that $W'',\scriptv_0$ together span $\scriptv_{n+1}$.
Since the three index sets $S',S'',\{0\}$ are disjoint,
general position implies that
$W',W''$ can be chosen so that $W''$ is transverse
to $W'+\scriptv_0$.
Thus $(\scriptv_j: 1\le n\le n+1)$
is a transverse splitting of $(\scriptv_j: 0\le j<n)$.

Let $(\kappa_j: 0\le j<n)$ satisfy \eqref{newhypothesis}, 
and choose $\kappa_n,\kappa_{n+1}$ as above.
We have proved that there exists an analytic subvariety
$Y_0\subset G(m,\kappa_j: 1\le j\le n+1)$ of positive codimension,
such that for any $\snarl^\sharp\in G(m,\kappa_j: 1\le j\le n+1)\subset Y_0$,
there exists at least one $\snarl\in G(m,\kappa_j: 0\le j<n)$
which admits at least one transverse splitting equal to $\snarl^\sharp$.
Indeed, each mention of ``general position'' in the above discussion
can be expressed as the condition that $\snarl^\sharp$ satisfies none
of a finite set of analytic equations.
The union of the varieties defined by each of these equations
defines $Y_0$, which has positive codimension.

Given any analytic subvariety 
$Y\subset G(m,\kappa_j: 1\le j\le n+1)$ of positive codimension,
set $\tilde Y=Y\cup Y_0$
and let $X$ be the set of all 
$\snarl\in G(m,\kappa_j: 0\le j<n)$ 
for which the subspaces $W'(\snarl),W''(\snarl)$ defined in the proof
of Lemma~\ref{lemma:constructW} either fail to define
a transverse splitting of $\snarl$, or define a splitting which belongs to $\tilde Y$.
Then $X$ is an analytic subvariety, for all restrictions encountered
can be expressed as analytic equations for $(\scriptv_j: 0\le j<n)$
together with the subspaces $U',U''$ used to define the functionals $W'(\cdot),W''(\cdot)$

We have shown that there exists at least one 
$\snarl_0\in G(m,\kappa_j: 0\le j<n)$
which admits some transverse splitting 
$\snarl_0^\sharp\in G(m,\kappa_j: 1\le j\le n+1)\setminus\tilde Y$. 
The subspaces $U',U''$
may be chosen so that $W'(\snarl_0),W''(\snarl_0)$ define this splitting
$\snarl_0^\sharp$. 
Then $X$ is nonempty, so $X$ has positive codimension.
\end{proof}

\section{The inductive step} \label{section:slicing}
We now prove Proposition~\ref{prop:induction}.
Let $W',W''$ 
and $\scriptv_n=\scriptv_0+W''$, $\kappa_n=\kappa''$, 
$\scriptv_{n+1}=\scriptv_0+W'$, and $\kappa_{n+1}=\kappa'$
be as in Lemma~\ref{lemma:constructW}.

Set $W=W'+W''$,
and $W^\star =\scriptv_0=\kernel(\pi_0)$.
$W,W^\star$ are a pair of supplementary subspaces, so 
$\reals^m=W+W^\star$ may be identified with $W\times W^\star$.
Thus an arbitrary element of $\reals^m$ can be expressed
in a unique way as $x+y$
with $x\in W$ and $y\in W^\star$; $x+y$ will henceforth be identified with $(x,y)$.

Define linear transformations $\tilde\pi_j:W\mapsto\reals^{\kappa_j}$
by \[\tilde\pi_j(x)=\pi_j(x,0).\]
For any $(x,y)$, $\pi_j(x,y) = \pi_j(x,0)+\pi_j(0,y)$,
so \[f_j(\pi_j(x,y))=f_{j,y}(\tilde\pi_j(x))\]
where \[f_{j,y}(t) = f_j(t+\pi_j(0,y)).\] 


We will use the equivalence, with the mappings $\pi_j$, sets $B_j$,
and phase function $P$ fixed, between an {\it a priori} inequality of the form 
\begin{equation}
|I(P;f_0,\cdots,f_{n-1})|\le  \scriptc\prod_{j=0}^{n-1}\norm{f_j}_\infty,
\end{equation}
and the formally stronger inequality
\begin{equation} \label{L2inequality}
|I(P;f_0,\cdots,f_{n-1})|\le \tilde\scriptc\norm{f_0}_2\prod_{j=1}^{n-1}\norm{f_j}_\infty. 
\end{equation}
If the latter holds, then the former holds with $\scriptc\le C\tilde\scriptc$.
If the former holds, then the latter follows with $\tilde\scriptc\le C\scriptc^{1/2}$,
by interpolation with the trivial inequality
\begin{equation}
|I(P;f_0,\cdots,f_{n-1})|\le \tilde C'\norm{f_0}_1\prod_{j=1}^{n-1}\norm{f_j}_\infty. 
\end{equation}
Our argument is not phrased exclusively in terms of one inequality or the other,
but uses their equivalence at each step of an induction.

Our oscillatory integral may be written as
\[
I(P;f_0,\cdots,f_{n-1})
=\int_{E}
I_y(P_y;f_{0,y},\cdots,f_{n-1,y})\,dy
\]
for some bounded subset $E\subset W^\star$,
where
\[
I_y(P_y;g_{0},\cdots,g_{n-1})
= \int e^{iP(x,y)}\prod_{j=0}^{n-1} g_j(\tilde\pi_j(x))\,dx.
\]

Note that 
\[f_{0,y}(\tilde\pi_0(x)) = f_0(\pi_0(x,0)+\pi_0(0,y)) \equiv f_0(\pi_0(x,0)).\]
$x\mapsto\pi_0(x,0)$ is a linear isomorphism of $W$ with $\reals^{\kappa_0}$.
Therefore by a linear change of variables in $\reals^{\kappa_0}$, we may arrange that 
\[\pi_0(x,0)\equiv x.\]
With this simplification,
\[
I_y(P_y;f_{0,y},\cdots,f_{n-1,y})
= \big\langle 
e^{iP(x,y)}\prod_{j=1}^{n-1} f_{j,y}(\tilde\pi_j(x)), \,
\overline{f_{0}} 
\big\rangle,
\]
where the inner product is taken with respect to $x$ for fixed $y$.

Fix bounded sets $B_j\subset\reals^{\kappa_j}$,
and consider only functions $f_j$ supported in $B_j$.
Define $\Lambda=\Lambda(P,\{\pi_j\})$ to be the optimal constant in the inequality
\eqref{L2inequality}.
Let $\{f_j: 1\le j\le n-1\}$  and $f_0$ be functions satisfying $\norm{f_j}_\infty\le 1$ for $j\ge 1$,
and $\norm{f_0}_2=1$,
such that
\[
|I(P;f_0,\cdots,f_{n-1})|\ge \tfrac12 \Lambda\norm{f_0}_2.
\]

There exists $z$ such that
$\big|\big \langle 
e^{iP(x,z)}\prod_{j=1}^{n-1} f_{j,z}(\tilde\pi_j(x)), \,
\overline{f_0} 
\big\rangle\big|\ge c\Lambda$.
Decompose \[f_0(x) = a e^{-iP(x,z)}\prod_{j=1}^{n-1} h_j(\tilde\pi_j(x)) +g_0(x)\]
where 
\begin{align*}
|a|&\le C\norm{f_0}_2
\\
\norm{g_0}_{\lt(\reals^{\kappa_0})}^2&\le \norm{f_0}_2^2 - c\Lambda^2\norm{f_0}_2^2
\end{align*}
and $h_j=\overline{f_{j,z}}$.
Then
\begin{multline*}
I(P;f_0,\cdots,f_{n-1})
=
I(P;f_1,\cdots,f_{n-1},g_0)
\\
+
a\iint e^{iP(x,y)} \prod_{j=1}^{n-1}f_j(\pi_j(x,y))
\cdot e^{-iP(x,z)}
\prod_{k=1}^{n-1} h_j(\tilde\pi_j(x))
\,dx\,dy.
\end{multline*}
The second term may be written as
\[
a
\iint e^{iQ(x,y)} \prod_{j=1}^{n-1}f_j(\pi_j(x,y))
\cdot
\prod_{k=1}^{n-1} h_k(\pi_k^\sharp(x,y))
\,dx\,dy
\]
where 
$\pi_k^\sharp:\reals^m\to\reals^{\kappa_k}$
is defined by
\[\pi_k^\sharp(x,y) = \tilde\pi_k(x) =\pi_k(x,0) = \pi_k(\pi_0(x,y),0)
= \pi_k(\pi_0(x,0),0)\]
and \[Q(x,y) = P(x,y)-P(x,z).\]

Since $x=\pi_0(x,y)$, 
$(x,y)\mapsto P(x,z)$ is a polynomial function of $\pi_0(x,y)$.
Therefore 
$[Q]=[P]$,
where $[\cdot]$ denotes the equivalence class in the space
of polynomials modulo those polynomials which are degenerate
relative to $\{\pi_j: 0\le j\le n-1\}$.

Now
\begin{multline*}
\iint e^{iQ(x,y)} \prod_{j=1}^{n-1}f_j(\pi_j(x,y))
\cdot
\prod_{k=1}^{n-1} h_k(\pi_k^\sharp(x,y))
\,dx\,dy
\\
=
I\Big(Q;f_1,\cdots,f_{n-1},h_1,\cdots,h_{n-1},
\{\pi_j\}_{j=1}^{n-1},\{\pi^\sharp_k\}_{k=1}^{n-1}\Big).
\end{multline*}
This is not what we are aiming for; for instance, this expression
is $2n-2$--multilinear, while we are aiming for an $n+1$--multilinear form.

Elements $(x,0)\in W$ may be decomposed
as $(x,0)=(x',x'',0)$ where $(x',0,0)\in W'$ and $(0,x'',0)\in W''$.
Thus
$\pi_k^\sharp(x',x'',y)=\pi_k(x',x'',0)= \pi_k(x',0,0)+\pi_k(0,x'',0)$ 
depends only on $x''$ for $k\in S'$,
and 
depends only on $x'$ for $k\in S''$;
the nullspace of $\pi_k^\sharp$ contains $W''+\scriptv_0$ for each $k\in S'$.
Therefore we may write
\[
\prod_{k\in S'}h_k(\pi_k^\sharp)(x',x'',y)
= f_n\big(\pi_n(x',x'',y)\big)
\]
where $\pi_n$ is a surjective linear mapping 
from $\reals^m$ to a Euclidean space of dimension
$\kappa_n=\dimension(W'')$,
the nullspace of $\pi_n$ equals $\scriptv_0+W''=\scriptv_n$,
\[
\pi_n(x',x'',y) = \pi_k^\sharp(x',x'',y) = \pi_k^\sharp(x',x'',0)
= \pi_k^\sharp(0,x'',0),
\]
and $\norm{f_n}_\infty \le\prod_{k\in S'}\norm{h_k}_\infty\le 1$;
this can be done, albeit in an artificial way, even if the intersection of the nullspaces
of all such $\pi_k^\sharp$ has dimension strictly greater than
$m-\kappa_n$, by defining $f_n$ to be independent of one or more coordinates in $\pi_n(\reals^m)$
in a sufficiently large bounded set.
Likewise
\[
\prod_{k\in S''}h_k(\pi_k^\sharp)(x',x'',y)
= f_{n+1}\big(\pi_{n+1}(x',x'',y)\big)
\]
where $\pi_{n+1}$ is a surjective linear mapping 
with nullspace $\scriptv_{n+1}$
from $\reals^m$ to a Euclidean space of dimension
$\kappa_{n+1}=\dimension(W')$,
and $\norm{f_{n+1}}_\infty\le 1$.
With these definitions,
\begin{multline*}
\iint e^{iQ(x,y)} \prod_{j=1}^{n-1}f_j(\pi_j(x,y))
\cdot
\prod_{k=1}^{n-1} h_k(\pi_k^\sharp(x))
\,dx\,dy
=
I\big(Q;f_1,\cdots,f_{n+1},
\{\pi_j\}_{j=1}^{n+1}\big).
\end{multline*}
$\norm{f_j}_\infty\le 1$ for all $j\in\{1,2,\cdots,n+1\}$,
and $f_i$ is supported in a bounded subset of $\reals^{\kappa_i}$
which depends only on $\{B_j: 0\le j\le n-1\}$,
on $\{\pi_j: 0\le j\le n-1\}$,
on the choices of $S',S''$, and on the choice of $W$.

Now $Q$ is 
nondegenerate\footnote{In fact, $Q$
is nondegenerate  relative to $\{\pi_j\}_{j=1}^{n+1}$,
if and only if $P$ is nondegenerate 
relative to $\{\pi_j\}_{j=0}^{n-1}$.}
relative to $\{\pi_j\}_{j=1}^{n+1}$,
because $Q$ is 
nondegenerate relative to $\{\pi_j\}_{j=0}^{n-1}$
and the projections $\pi_n,\pi_{n+1}$ both factor through $\pi_0$.
The norm of $Q$ in the quotient space of polynomials
modulo sums of polynomials $q\circ\pi_j$ with $1\le j\le n+1$
is at least as large as the norm of $P$ in the 
quotient space of polynomials modulo $q\circ\pi_j$ with $0\le j\le n-1$,
up to a constant factor which depends only on choices of norms for these spaces.

We are reasoning under the induction hypothesis that
for any collection of bounded subsets $B_j\subset \reals^{\kappa_j}$,
there exist $C<\infty$ and an exponent $\gamma>0$ such that 
for all continuous functions $f_j$ supported in $B_j$ respectively,
\begin{equation}
\big| I\big(Q;f_1,\cdots,f_{n+1},
\{\pi_j\}_{j=1}^{n+1}\big)\big|
\le C\langle
\norm{Q}_{\text{ND}}
\rangle^{-\gamma}
\prod_{j=1}^{n+1}\norm{f_j}_\infty
\le C\langle
\norm{P}_{\text{ND}}
\rangle^{-\gamma}.
\end{equation}
$C,\gamma$
depend on $\{\pi_j: 1\le j\le n+1\}$ and on $\{B_j\}$,
which in turn depend on $\{\pi_j: 0\le j\le n-1\}$
and on the designation of bounded subsets on which
the functions $f_j$ are supported for all $j\in\{0,1,\cdots,n-1\}$.

Therefore whenever $\{f_j\}$ are continuous functions
supported in $B_j\subset\reals^{\kappa_j}$,
satisfying $\norm{f_0}_2=1$
and $\norm{f_j}_\infty\le 1$ for all $j\in\{1,2,\cdots,n-1\}$,
\begin{align*}
|I(f_0,\cdots,f_{n-1},&\{\pi_j\}_{j=0}^{n-1})|
\\
&\le
|I(g_0,f_1,\cdots,f_{n-1},\{\pi_j\}_{j=0}^{n-1})|
+C 
\big| I\big(Q;f_1,\cdots,f_{n+1},
\{\pi_j\}_{j=1}^{n+1}\big)\big|
\\
&\le
\Lambda\norm{g_0}_2
+ 
C\langle
\norm{P}_{\text{ND}}
\rangle^{-\gamma}
\\
&\le
\Lambda \norm{f_0}_2(1-c\Lambda^2)
+ 
C\langle
\norm{P}_{\text{ND}}
\rangle^{-\gamma}
\\
&\le
\Lambda (1-c\Lambda^2)
+ 
C\langle
\norm{P}_{\text{ND}}
\rangle^{-\gamma}.
\end{align*}

By taking the supremum over all $f_0,\cdots,f_{n-1}$
which are supported in the sets $B_j$ and satisfy
$\norm{f_0}_2\le 1$ and $\norm{f_j}_\infty=1$
for all $j\ge 1$, we conclude that
\begin{equation*}
\Lambda\le
\Lambda (1-c\Lambda^2)
+ 
C\langle
\norm{P}_{\text{ND}}
\rangle^{-\gamma}.
\end{equation*}
Subtracting $\Lambda$ from both sides and rearranging yields
\begin{equation} \label{swallowingresult}
\Lambda^3
\le C\langle
\norm{P}_{\text{ND}}
\rangle^{-\gamma}.
\end{equation}
This completes the proof of Proposition~\ref{prop:induction},
hence of Theorem~\ref{thm:reduction}.

\end{document}